\documentclass[11pt,a4paper]{article}
\usepackage{amsfonts,amsgen,amsmath,amsthm,caption,color,eepic,graphicx,indentfirst,mathrsfs,multirow,tikz,verbatim}

\theoremstyle{definition}
\newtheorem{case}{Case}
\newtheorem{subcase}{Case}[case]

\newtheorem{claim}{Claim}

\newtheorem{construction}{Construction}
\newtheorem{corollary}{Corollary}

\newtheorem{lemma}{Lemma}
\newtheorem{proposition}{Proposition}

\newtheorem{theorem}{Theorem}

\newcounter{mathitem}
\newenvironment{mathitem}
  {\begin{list}{{$(\roman{mathitem})$}}{
   \setcounter{mathitem}{0}
   \usecounter{mathitem}
   \setlength{\topsep}{0pt plus 2pt minus 0pt}
   \setlength{\parskip}{0pt plus 2pt minus 0pt}
   \setlength{\partopsep}{0pt plus 2pt minus 0pt}
   \setlength{\parsep}{0pt plus 2pt minus 0pt}
   \setlength{\leftmargin}{25pt}
   \setlength{\itemsep}{0pt plus 2pt minus 0pt}}}
  {\end{list}}

\usetikzlibrary{shapes.geometric, positioning} 

\newcommand{\drawCfive}[8]{
  \begin{scope}[shift={(#1, #2)}] 
  \draw (90:2) -- (162:2) -- (234:2) -- (306:2) -- (18:2) -- cycle; 
  \node at (90:2) [circle, fill=#4, fill opacity=1, inner sep=2pt] {}; 
  \node at (162:2) [circle, fill=#5, fill opacity=1, inner sep=2pt] {};
  \node at (234:2) [circle, fill=#6, fill opacity=1, inner sep=2pt] {};
  \node at (306:2) [circle, fill=#7, fill opacity=1, inner sep=2pt] {};
  \node at (18:2) [circle, fill=#8, fill opacity=1, inner sep=2pt] {};
  \node at (90:2.5) {$v_1$}; 
  \node at (162:2.5) {$v_2$};
  \node at (234:2.5) {$v_3$};
  \node at (306:2.5) {$v_4$};
  \node at (18:2.5) {$v_5$};
  \node at (0, -3) {#3}; 
  \end{scope}
}

\newcommand{\drawCsix}[9]{
  \begin{scope}[shift={(#1, #2)}] 
  \draw (90:1.8) -- (150:1.8) -- (210:1.8) -- (270:1.8) -- (330:1.8) -- (30:1.8) -- cycle; 
  \node at (90:1.8) [circle, fill=#4, fill opacity=1, inner sep=2pt] {}; 
  \node at (150:1.8) [circle, fill=#5, fill opacity=1, inner sep=2pt] {};
  \node at (210:1.8) [circle, fill=#6, fill opacity=1, inner sep=2pt] {};
  \node at (270:1.8) [circle, fill=#7, fill opacity=1, inner sep=2pt] {};
  \node at (330:1.8) [circle, fill=#8, fill opacity=1, inner sep=2pt] {};
  \node at (30:1.8) [circle, fill=#9, fill opacity=1, inner sep=2pt] {};
  \node at (90:2.3) {$v_1$}; 
  \node at (150:2.3) {$v_2$};
  \node at (210:2.3) {$v_3$};
  \node at (270:2.3) {$v_4$};
  \node at (330:2.3) {$v_5$};
  \node at (30:2.3) {$v_6$};
  \node at (0, -3) {#3}; 
  \end{scope}
}

\begin{document}

\title{\bf\Large Berge tight cycles of all lengths in hypergraphs\thanks{Supported by NSFC (Nos. 12171393, 12571378).}
}
\author{
Minghui Yu$^{a,b}$, Binlong Li$^{a,b}$, Ruonan Li$^{a,b,}$\thanks{Corresponding author. E-mail addresses: mhyu@mail.nwpu.edu.cn (M. Yu), binlongli@nwpu.edu.cn (B. Li), rnli@nwpu.edu.cn (R. Li)}~\\[2mm]
\small $^{a}$School of Mathematics and Statistics, \\
\small Northwestern Polytechnical University, Xi'an, Shaanxi 710129, China\\
\small $^{b}$Xi'an-Budapest Joint Research Center for Combinatorics, \\
\small Northwestern Polytechnical University, Xi'an, Shaanxi 710129, China
}

\date{}

\maketitle

\begin{abstract}
Given a set $R$ of positive integers, an $R$-graph $H = (V, E)$ is a hypergraph where the cardinality of each hyperedge belongs to $R$. If $R = \{r\}$, we sometimes refer to the hypergraph as an $r$-graph rather than an $R$-graph. For a set $S \subseteq V$, let $d_H(S)$ denote the number of hyperedges of $H$ containing $S$. Given a nonnegative integer $s$, the minimum $s$-degree $\delta_s(H)$ is the minimum of $d_H(S)$ over all $s$-vertex subsets $S$ of $V$. Let $r$ and $t$ be positive integers with $r < t$. We denote by $C_t^r$ the $t$-vertex $r$-uniform tight cycle, which is an $r$-graph with at least three hyperedges whose vertices admit a cyclic ordering such that every $r$ consecutive vertices form a hyperedge. In particular, $C_t^2$ is the classical cycle $C_t$ in $2$-graphs. For hypergraphs $F$ and $H$, we say that $H$ is a Berge-$F$ if there exist an injection $f \colon V(F) \to V(H)$ and a bijection $g \colon E(F) \to E(H)$ such that $\{f(v): v \in e\} \subseteq g(e)$ for all $e \in E(F)$.

Lu and Wang [Discrete Math. 344 (2021), 112462] proved that every $[3]$-graph $H$ on $n \geq 6$ vertices with $\delta_2(H) \geq 1$ contains a Berge-$C_t$ for all $3 \leq t \leq n$. In this paper, we prove that for any positive integer $r$ and any set $R \subseteq [k]$ with $k \geq 2$, there exists an integer $n_0 = n_0(k,r)$ such that every $R$-graph $H$ on $n \geq n_0$ vertices with $\delta_r(H) \geq 1$ contains a Berge-$C_t^r$ for all $r+1 \leq t \leq n$. In particular, when $k = 4$ and $r = 3$, we show that every $[4]$-graph $H$ on $n \geq 9$ vertices with $\delta_3(H) \geq 1$ contains a Berge-$C_t^3$ for all $4 \leq t \leq n$. We also characterize all the counterexamples when $4 \leq n \leq 8$.

\smallskip
\end{abstract}

\section{Introduction}\label{S1}
Given a set $R$ of positive integers, an \emph{$R$-uniform hypergraph} ($R$-graph) $H = (V, E)$ consists of a vertex set $V$ and a hyperedge set $E$ where the cardinality of each hyperedge belongs to $R$. If $R = \{r\}$, we sometimes refer to the hypergraph as an $r$-graph rather than an $\{r\}$-graph. By convention, a simple graph corresponds to a $2$-graph. A hyperedge is said to be a $k$-hyperedge if it has precisely $k$ vertices. We use $K_n^r$ to denote the \emph{complete $r$-graph} on $n$ vertices, whose hyperedge set comprises all $r$-element subsets of its vertex set. For an $R$-graph $H$ and a set $S \subseteq V$, let $d_H(S)$ denote the number of hyperedges of $H$ containing $S$. Given a nonnegative integer $s$, the minimum $s$-degree $\delta_s(H)$ is the minimum of $d_H(S)$ over all $s$-vertex subsets $S$ of $V$. Note that $\delta_0(H)$ is the number of hyperedges of $H$. For positive integers $a \leq b$, we write $[a,b] = \{a, a+1, \dots, b\}$ and $[b] = [1,b]$ for simplicity.

There are many distinct notions of cycles in hypergraphs. Let $\ell$, $r$ and $t$ be positive integers with $\ell < r < t$. The $t$-vertex $r$-uniform \emph{$\ell$-cycle}, denoted $C_t^{r,\ell}$, is an $r$-graph with at least three hyperedges whose vertices have a cyclic ordering such that each hyperedge consists of exactly $r$ consecutive vertices and any two consecutive hyperedges intersect in precisely $\ell$ vertices. Such a cycle is called a \emph{loose cycle} when $\ell = 1$, and a \emph{tight cycle} when $\ell = r - 1$. For $\ell = r - 1$, we abbreviate $C_t^{r, r-1}$ to $C_t^r$. In particular, $C_t^2$ is the classical cycle $C_t$ in $2$-graphs. The \emph{length} of an $\ell$-cycle is defined as the number of hyperedges that it contains. An $r$-graph $H$ is said to contain a \emph{Hamilton $\ell$-cycle} if it has an $\ell$-cycle that spans all vertices of $H$. Since an $r$-uniform $\ell$-cycle on $t$ vertices contains exactly $t/(r-\ell)$ hyperedges, a necessary condition for the existence of a Hamilton $\ell$-cycle in an $n$-vertex $r$-graph is that $(r-\ell)$ divides $n$.

Gerbner and Palmer \cite{GP} introduced the notion of Berge hypergraphs, a generalization of the well-known Berge paths and Berge cycles originally defined in \cite{B}. For hypergraphs $F$ and $H$, we say that $H$ is a \emph{Berge-$F$} if there exist an injection $f \colon V(F) \to V(H)$ and a bijection $g \colon E(F) \to E(H)$ such that $\{f(v): v \in e\} \subseteq g(e)$ for every $e \in E(F)$. The vertices in $\{f(v): v \in V(F)\}$ are called the \emph{defining vertices} of the Berge-$F$. Given a hypergraph $F$, there can be multiple Berge-$F$, and $F$ itself is a Berge-$F$. In particular, a \emph{Berge-$C_t^r$} consists of a sequence of $t$ distinct vertices $v_1, v_2, \dots, v_t$ and a sequence of $t$ distinct hyperedges $h_1, h_2, \dots, h_t$, where the following condition holds: for each $i \in [t]$, $\{v_i, v_{i+1}, \dots, v_{i+r-1}\} \subseteq h_i$ (indices are taken modulo $t$). Furthermore, we say that the vertex set $\{v_i, v_{i+1}, \dots, v_{i+r-1}\}$ is \emph{embedded} in the hyperedge $h_i$ when $\{v_i, v_{i+1}, \dots, v_{i+r-1}\} \subseteq h_i$. A hypergraph $H$ is said to \emph{contain} a Berge-$F$ if some Berge-$F$ is a subhypergraph of $H$; otherwise, $H$ is called \emph{Berge-$F$-free}. The Fano plane $\mathbf{F}$ is the unique $3$-graph with seven vertices and seven hyperedges such that every pair of vertices is contained in exactly one hyperedge.

In 1952, Dirac \cite{D} proved that every graph $G$ on $n \geq 3$ vertices with $\delta_1(G) \geq n/2$ contains a Hamilton cycle. Katona and Kierstead \cite{KK} first investigated the minimum $r$-degree condition for Hamilton $\ell$-cycles in hypergraphs. Over the past two decades, Dirac-type conditions for the existence of Hamilton $\ell$-cycles have been extensively studied; we refer the reader to the recent surveys \cite{RR, Z} for further results. In particular, Reiher et al. \cite{RRRSS} extended Dirac's theorem to $3$-graphs and proved that every $3$-graph $H$ on sufficiently large $n$ vertices with $\delta_1(H) \geq (5/9+o(1))\binom{n}{2}$ contains a tight cycle $C_n^3$.

The Tur\'an numbers for Berge paths and Berge cycles have been extensively studied. We refer the reader to recent results in \cite{GLSZ, GS, GSZ}. Bermond et al. \cite{BGHS} first investigated the minimum $1$-degree condition in $r$-graphs, and this line of research was subsequently improved in \cite{CEP, CP, MHG}. Kostochka et al. \cite{KLM} determined the optimal minimum $1$-degree condition for $r$-graphs that forces Berge-$C_n$. F\"uredi et al. \cite{FKL} obtained the analogous condition in non-uniform hypergraphs. In addition, Halfpap and Magnan \cite{HM} established the minimum positive codegree threshold for Hamilton Berge cycles in $3$-graphs. Focusing on a hypergraph $H$ with $\delta_2(H) \ge 1$, Lu and Wang \cite{LW} studied the existence of Berge-$C_t$.

\begin{theorem}[Lu et al. \cite{LW}] \label{T1}
For a positive set $R \subseteq [k]$ with $k \geq 2$, there exists an integer $n_0 = n_0(k, r)$ such that every $R$-graph $H$ on $n \geq n_0$ vertices with $\delta_2(H) \geq 1$ contains a Berge-$C_t$ for all $3 \leq t \leq n$.
\end{theorem}

\begin{theorem}[Lu et al. \cite{LW}] \label{T2}
Every $[3]$-graph $H$ on $n \geq 6$ vertices with $\delta_2(H) \geq 1$ contains a Berge-$C_t$ for all $3 \leq t \leq n$.
\end{theorem}

Lu and Wang \cite{LW1, LW2} also investigated Ramsey-type and Tur\'an-type problems for Berge-$G$, where $G$ is a $2$-graph. In fact, only a few studies concern the existence of Berge-$F$ when $F$ is a hypergraph. See \cite{DGS, GSS1, GSS2, MO} for Ramsey-type results concerning Berge-$C_t^r$ with $r \geq 3$, and \cite{AS, BGKKP} for Tur\'an-type results. In this paper, we consider the existence of Berge-$C_t^r$ in hypergraphs for $r \geq 3$, which generalizes the results of Theorems \ref{T1} and \ref{T2}.

\begin{theorem}\label{T0}
For a positive integer $r$ and a set $R \subseteq [k]$ with $k \geq 2$, there exists an integer $n_0 = n_0(k, r)$ such that every $R$-graph $H$ on $n \geq n_0$ vertices with $\delta_r(H) \geq 1$ contains a Berge-$C_t^r$ for all $r+1 \leq t \leq n$.
\end{theorem}

In particular, when $k = 4$ and $r = 3$, we show that every $[4]$-graph $H$ on $n \geq 9$ vertices with $\delta_3(H) \geq 1$ contains a Berge-$C_t^3$ for all $4 \leq t \leq n$. We also characterize all the counterexamples when $4 \leq n \leq 8$.

\begin{construction}\label{Construction1}
Let $\mathcal{H}_4 = \{H_4, H_4', H_4''\}$ be a family of hypergraphs on $[4]$ such that
\begin{mathitem}
\item[$(1)$] $E(H_4) = \{\{1, 2, 3, 4\}\}$;
\item[$(2)$] $E(H_4') = \{\{1, 2, 3, 4\}, \{1, 2, 3\}\}$;
\item[$(3)$] $E(H_4'') = \{\{1, 2, 3, 4\}, \{1, 2, 3\}, \{1, 2, 4\}\}$.
\end{mathitem}
\end{construction}

\begin{construction}\label{Construction2}
Let $\mathcal{H}_5 = \{H_5, H_5', H_5'', H_5'''\}$ be a family of hypergraphs on $[5]$ such that
\begin{mathitem}
\item[$(1)$] $E(H_5) = \{\{1, 2, 3, 4\}\} \cup \{\{i, j, 5\}: \{i, j\} \in \binom{[4]}{2}\}$;
\item[$(2)$] $E(H_5') = \{\{1, 2, 3, 4\}, \{1, 2, 3, 5\}, \{1, 4, 5\}, \{2, 4, 5\}, \{3, 4, 5\}\}$;
\item[$(3)$] $E(H_5'') = \{\{1, 2, 3, 4\}, \{1, 2, 3, 5\}, \{1, 2, 4, 5\}, \{3, 4, 5\}\}$;
\item[$(4)$] $E(H_5''') = \{\{1, 2, 3, 4\}, \{1, 2, 3, 5\}, \{1, 2, 4, 5\}, \{1, 3, 4, 5\}\}$.
\end{mathitem}
\end{construction}

\begin{construction}\label{Construction3}
Let $H_6$, $H_7$, $H_8$ be hypergraphs such that
\begin{mathitem}
\item[$(1)$] $V(H_6) = [6]$ and $E(H_6) = \{\{1, 2, 3, 4\}, \{1, 2, 5, 6\}, \{3, 4, 5, 6\}\} \cup \{\{i, j, k\}: i \in \{1, 2\}, j \in \{3, 4\}, k \in \{5, 6\}\}$;
\item[$(2)$] $V(H_7) = [7]$ and $E(H_7) = \{e, [7] \setminus e:  e \in E(\mathbf{F})\}$;
\item[$(3)$] $V(H_8) = [8]$ and $E(H_8) = \{e \cup \{8\}, [7] \setminus e: e \in E(\mathbf{F})\}$,
\end{mathitem}
where $\mathbf{F}$ denotes a Fano plane on $[7]$.
\end{construction}

\begin{theorem}\label{T}
Every $\{3, 4\}$-graph $H$ on $n \geq 4$ vertices with $\delta_3(H) \geq 1$ contains a Berge-$C_t^3$ for all $4 \leq t \leq n$, unless $H \in \mathcal{H}_4$ and $t = 4$ or $H \in \mathcal{H}_5 \cup \{H_6, H_7, H_8\}$ and $t = 5$.
\end{theorem}

We remark that under the condition $\delta_3(H) \geq 1$, it is sufficient to consider $\{3, 4\}$-graphs rather than $[4]$-graphs. Thus Theorem \ref{T} implies that every $[4]$-graph $H$ on $n \geq 9$ vertices with $\delta_3(H) \geq 1$ contains a Berge-$C_t^3$ for all $4 \leq t \leq n$. For $4 \leq n \leq 8$, all the counterexamples in a $[4]$-graph are obtained by adding arbitrary $1$-hyperedges and $2$-hyperedges to a hypergraph $H$ from Constructions \ref{Construction1}--\ref{Construction3}. When $H$ is a $4$-graph, we immediately obtain the following corollary from Theorem \ref{T}.

\begin{corollary}\label{TC}
Every $4$-graph $H$ on $n \geq 4$ vertices with $\delta_3(H) \geq 1$ contains a Berge-$C_t^3$ for all $4 \leq t \leq n$, unless $H \in \{H_4, H_5''', H_8\}$.
\end{corollary}

\section{Proof of Theorem \ref{T0}}\label{S2}
For a $[k]$-graph $H = (V, E)$ and a subset $S \subseteq V$, the \emph{$c$-trace} of $H$ on $S$ is the $[k]$-graph $H_S^c = (S, E')$, where the vertex set is $S$ and the hyperedge set $E'$ is defined as $E' = \{h \cap S: h \in E(H), |h \cap S| \geq c\}$. Trace operations play a crucial role in extremal problems involving non-uniform hypergraphs. The following propositions can be readily verified by definition (see \cite{LW}).

\begin{proposition}\label{P}
Let $H$ be a $[k]$-graph and $S \subseteq V(H)$. Then the following statements hold:
\begin{mathitem}
\item[$(1)$] If $\delta_c(H) \geq 1$ and $|S| \geq c$, then $\delta_c(H_S^c) \geq 1$.
\item[$(2)$] Each Berge-$C_t^r$ in $H_S^c$ corresponds to a Berge-$C_t^r$ in $H$ with the same defining vertices.
\end{mathitem}
\end{proposition}

A \emph{hyperedge-coloring} of a hypergraph $H$ is defined as a mapping $\phi: E(H) \to \mathbb{N}$, where $\mathbb{N}$ denotes the set of natural numbers. We call a hypergraph $H$ with a hyperedge-coloring $\phi$ a \emph{hyperedge-colored hypergraph}. A hyperedge-colored hypergraph $F$ is said to be \emph{rainbow} if each hyperedge in $E(F)$ is assigned a distinct color. We say that a hyperedge-coloring $\phi$ is \emph{$r$-bounded} if each color is used at most $r$ times. Dudek et al. \cite{DFR} initiated the study of rainbow Hamilton cycles in complete $r$-graphs using the Lov\'asz local lemma. These results were later strengthened by Dudek and Ferrara \cite{DF}.

\begin{theorem}[Dudek, Frieze and Ruci\'nski \cite{DFR}]\label{T3}
Given integers $\ell$, $r$ with $1 \leq \ell < r$, we can find constants $n_0 = n_0(r, \ell)$ and $c = c(r, \ell)$ such that if $n \geq n_0$ and $r - \ell$ divides $n$, then every $cn^{r-\ell}$-bounded hyperedge-coloring of $K_n^r$ contains a rainbow copy of $C_n^{r, \ell}$.
\end{theorem}

\begin{lemma}\label{L21}
For a positive integer $r$ and a set $R \subseteq [k]$ with $k \geq 2$, there exists an integer $N = N(k, r)$ such that every $R$-graph $H$ on $n \geq N$ vertices with $\delta_r(H) \geq 1$ contains a Berge-$C_n^r$.
\end{lemma}

\begin{proof}
Let $E(H) = \{e_1, e_2, \dots, e_m\}$. We construct a hyperedge-coloring $\phi$ of the auxiliary $r$-graph $G = K_n^r$ from $H$ by assigning to each hyperedge $h \in E(G)$ the color $c_i$ when $h \subseteq e_i$ for some $i$, and choosing a color arbitrarily if $h$ is contained in multiple hyperedges of $H$. Observe that $\phi$ is $\binom{k}{r}$-bounded. By Theorem \ref{T3}, there exists an integer $N$ such that $G$ contains a rainbow copy of $C_n^r$ for $n \geq N$. This rainbow copy corresponds to a Berge-$C_n^r$ in $H$, where each hyperedge $h \in E(C_n^r)$ is embedded into $e_i$ whenever $h$ is colored $c_i$.
\end{proof}

\begin{lemma}\label{L22}
Given positive integers $r < s$ and a subset $R\subseteq [k]$ with $k \geq 2$, there exists an integer $n_0 = n_0(k, r, s)$ such that every $R$-graph $H$ on $n \geq n_0$ vertices with $\delta_r(H) \geq 1$ contains a Berge-$K_s^r$.
\end{lemma}

\begin{proof}
We assume that $H$ is hyperedge-minimal with respect to the condition $\delta_r(H) \geq 1$. Let $E(H) = \{e_1, e_2, \dots, e_m\}$. Since $\delta_r(H) \geq 1$, every $r$-vertex subset of $V(H)$ is contained in some hyperedge of $H$. Thus, we may assign a hyperedge to each $r$-vertex subset, which implies that $m \leq \binom{n}{r}$.

Let $S \subseteq V$ be a uniformly random subset of $V$ with $|S|=s$. For each $i \in [m]$, we use $B_i$ to denote the event $|e_i \cap S| \geq r+1$. It is straightforward to verify that
$$\Pr(B_i) \leq \frac{\binom{|e_i|}{r+1}\binom{n-r-1}{s-r-1}}{\binom{n}{s}} \leq \frac{\binom{k}{r+1}\binom{n-r-1}{s-r-1}}{\binom{n}{s}}.$$
Applying a union bound over all $B_i$, we get
$$\Pr(B_1 \vee \dots \vee B_m) \leq \sum_{i=1}^m \Pr(B_i) \leq m \frac{\binom{k}{r+1}\binom{n-r-1}{s-r-1}}{\binom{n}{s}} \leq \binom{n}{r} \frac{\binom{k}{r+1}\binom{n-r-1}{s-r-1}}{\binom{n}{s}} = o(1)$$
as $n \to \infty$. Thus, there exists an integer $n_0 = n_0(k, r, s)$ such that for any fixed $n \geq n_0$, there is a set $S$ satisfying $|e_i \cap S| \leq r$ for every $i \in [m]$.

Next, we claim that $H$ contains a Berge-$K_s^r$. Indeed, since $\delta_r(H) \geq 1$, for every hyperedge $e \in \binom{S}{r}$, there exists a hyperedge $g(e) \in E(H)$ such that $e \subseteq g(e)$. Furthermore, for any distinct $h_1, h_2 \in \binom{S}{r}$, we have $g(h_1) \neq g(h_2)$ because both $|g(h_1) \cap S|$ and $|g(h_2) \cap S|$ are at most $r$.
We obtain a Berge-$K_s^r$ in $H_S^r$, which implies that $H$ contains a Berge-$K_s^r$ with the same defining vertices by Proposition \ref{P}.
\end{proof}

\begin{proof}[Proof of Theorem \ref{T0}]
By Lemma \ref{L21}, there exists an integer $N = N(k, r)$ such that every $R$-graph $H$ on $n \geq N$ vertices with $\delta_r(H) \geq 1$ contains a Berge-$C_n^r$. Similarly, Lemma \ref{L22} implies that there exists an integer $n_0 = n_0(k, r, N)$ such that every $R$-graph $H$ on $n \geq n_0$ vertices with $\delta_r(H) \geq 1$ contains a Berge-$K_N^r$. Note that $n_0 \geq N$. Now let $H$ be an $R$-graph on $n \geq n_0$ vertices with $\delta_r(H) \geq 1$. Since $H$ contains a Berge-$K_N^r$, it follows that $H$ contains a Berge-$C_t^r$ for all $r+1 \leq t \leq N$. Moreover, for every subset $S \subseteq V(H)$ with $N+1 \leq |S| \leq n$, we have $\delta_r(H_S^r) \geq 1$ by Proposition \ref{P}. Thus the $r$-trace hypergraph $H_S^r$ contains a Berge-$C_{|S|}^r$. By Proposition \ref{P}, $H$ contains a Berge-$C_t^r$ for all $N+1 \leq t \leq n$. Consequently, every such $R$-graph $H$ contains a Berge-$C_t^r$ for every integer $t$ with $r+1 \leq t \leq n$.
\end{proof}

\section{Proof of Theorem \ref{T}}\label{S3}
Before proving Theorem \ref{T}, we first establish several auxiliary lemmas. Lemma \ref{L10} shows that the hypergraph given in Constructions \ref{Construction1}--\ref{Construction3} contains neither a Berge-$C_4^3$ nor a Berge-$C_5^3$. Furthermore, Lemmas \ref{L11}--\ref{L12} imply that $H$ always contains a Berge-$C_4^3$ for $n \geq 4$ and a Berge-$C_5^3$ for $n \geq 5$, except for the exceptional cases mentioned above. In addition, Lemma \ref{L} provides a key technical tool for extending the cycle length.

Recall that a Berge-$C_t^3$ is defined by a vertex sequence $v_1, v_2, \dots, v_t$ and a hyperedge sequence $h_1, h_2, \dots, h_t$ such that $\{v_i, v_{i+1}, v_{i+2}\}\subseteq h_i$ for all $i \in [t]$, where indices are taken modulo $t$.

\begin{lemma}\label{L10}
Every $H \in \mathcal{H}_4$ is Berge-$C_4^3$-free and every $H \in \mathcal{H}_5 \cup \{H_6, H_7, H_8\}$ is Berge-$C_5^3$-free.
\end{lemma}

\begin{proof}
We first consider $H \in \mathcal{H}_4$. Since a Berge-$C_4^3$ consists of four hyperedges and $H$ contains at most three hyperedges, $H$ is Berge-$C_4^3$-free. Similarly, $H_5''$ and $H_5'''$ are Berge-$C_5^3$-free. Now we suppose that $H_5$ contains a Berge-$C_5^3$. By symmetry among the vertices $1$, $2$, $3$, $4$, we may take its vertex sequence as $1, 2, 3, 4, 5$ and hyperedge sequence as $e_1, e_2, e_3, e_4, e_5$. This forces $e_1 = e_2 = \{1, 2, 3, 4\}$, a contradiction.

Next, we suppose that $H_5'$ contains a Berge-$C_5^3$ with the hyperedge sequence $e_1, e_2, e_3, e_4, e_5$. By symmetry among the vertices $1, 2, 3$ and the symmetry between vertex $4$ and vertex $5$, it suffices to consider the vertex sequences $1, 2, 3, 4, 5$ and $1, 2, 4, 3, 5$. For the vertex sequence $1, 2, 3, 4, 5$, we derive $e_2 = \{1, 2, 3, 4\}$ and $e_5 = \{1, 2, 3, 5\}$, which leaves no valid choice for $e_1$ and thus gives a contradiction. For the vertex sequence $1, 2, 4, 3, 5$, we obtain $e_5 = \{1, 2, 3, 5\}$, so no feasible hyperedge $e_4$ exists, a contradiction.

Finally, we take $H \in \{H_6, H_7, H_8\}$ and suppose $H$ contains a Berge-$C_5^3$ with defining vertices $S$. Since every triple of $V(H)$ is contained in exactly one hyperedge, $H_S^3$ also contains a Berge-$C_5^3$. However, we can check that any subset $S \subseteq V(H)$ with $|S| = 5$ satisfies $H_S^3 \cong H_5$, a contradiction.
\end{proof}

\begin{lemma}\label{L11}
Let $H$ be a $\{3, 4\}$-graph on $n \geq 4$ vertices. If $\delta_3(H) \geq 1$, then $H$ contains a Berge-$C_4^3$ unless $H \in \mathcal{H}_4$.
\end{lemma}

\begin{proof}
Let $V(H) = \{v_i: i \in [n]\}$. For $n = 4$, the hyperedges that may appear in $H$ are $e_0 = \{v_1, v_2, v_3, v_4\}$ and $e_i = e_0 \setminus \{v_i\}$ for each $i \in [4]$. If $H$ contains at least four hyperedges, then the vertex sequence $v_1, v_2, v_3, v_4$ and the hyperedge sequence $e_4', e_1', e_2', e_3'$ form a Berge-$C_4^3$, where $e_i' = e_i$ if $e_i \in E(H)$ and $e_i' = e_0$ otherwise. Hence, $H$ contains at most three hyperedges. Now we assert that $H$ must be Berge-$C_4^3$-free, which follows from the fact that a Berge-$C_4^3$ consists of four distinct hyperedges. Note that a $3$-hyperedge contains one triple, while a $4$-hyperedge contains four triples. Since there are $\binom{4}{3} = 4$ triples on $V(H)$ and $\delta_3(H) \geq 1$, we conclude that $e_0 \in E(H)$. Accordingly, $H = H_4$, $H_4'$ or $H_4''$ when it contains zero, one or two $3$-hyperedges, respectively.

For $n \geq 5$, if $H$ is a $3$-graph, then $\delta_3(H) \geq 1$ implies $H = K_n^3$. Hence $H$ contains a Berge-$C_4^3$. Next we assume that $H$ contains a $4$-hyperedge $h = \{v_1, v_2, v_3, v_4\}$. Since $\delta_3(H) \geq 1$, there exist hyperedges $h', h''$ such that $\{v_1, v_2, v_5\} \subseteq h'$ and $\{v_3, v_4, v_5\} \subseteq h''$. We claim that $h'' \cap \{v_1, v_2\} \neq \emptyset$. Suppose to the contrary that $h''$ is either $\{v_3, v_4, v_5\}$ or contains a vertex outside $\{v_i: i \in [5]\}$. Since $\delta_3(H) \geq 1$, there exist hyperedges $h_1, h_1'$ such that $\{v_1, v_3, v_5\} \subseteq h_1$ and $\{v_1, v_4, v_5\} \subseteq h_1'$. We have $h_1 = h_1' = \{v_1, v_3, v_4, v_5\}$; otherwise the vertex sequence $v_1, v_3, v_4, v_5$ and the hyperedge sequence $h, h'', h_1', h_1$ form a Berge-$C_4^3$. Similarly, there exists a hyperedge $h_2 = \{v_2, v_3, v_4, v_5\}$. But now the vertex sequence $v_1, v_2, v_3, v_5$ and the hyperedge sequence $h, h_2, h_1, h'$ form a Berge-$C_4^3$, a contradiction. By symmetry, $h' \cap \{v_3, v_4\} \neq \emptyset$. Without loss of generality, we assume $v_2 \in h''$ and $v_3 \in h'$. Since $\delta_3(H) \geq 1$, there exists a hyperedge $h_3$ such that $\{v_1, v_4, v_5\} \subseteq h_3$. Then the vertex sequence $v_1, v_4, v_2, v_5$ and the hyperedge sequence $h, h'', h', h_3$ form a Berge-$C_4^3$.
\end{proof}

\begin{lemma}\label{L13}
For any $H \in \mathcal{H}_5$, let $H'$ be a hypergraph obtained by adding a new $3$-hyperedge $e$ to $H$. Then $H'$ contains a Berge-$C_5^3$.
\end{lemma}
\begin{proof}
For convenience, let $e_{ijk}$ denote the $3$-hyperedge $\{i,j,k\}$ and $e_{ijkl}$ denote the $4$-hyperedge $\{i,j,k,l\}$. We first assume that $H = H_5$. Since all triples of $[5]$ that are not contained in $\{1,2,3,4\}$ already belong to $E(H_5)$, the newly added hyperedge $e$ must be a subset of $e_{1234}$. Without loss of generality, we assume $e = e_{123}$. Then the vertex sequence $1, 2, 3, 4, 5$ and the hyperedge sequence $e, e_{1234}, e_{345}, e_{145}, e_{125}$ form a Berge-$C_5^3$.

Next, we assume that $H = H_5'$. Since all triples of $[5]$ that are not contained in $\{1,2,3,4\}$ and $\{1,2,3,5\}$ already belong to $E(H_5)$, the newly added hyperedge $e$ must be a subset of $e_{1234}$ or $e_{1235}$. This implies that either $e = e_{123}$ or $e$ consists of two vertices from $\{1,2,3\}$ and one vertex from $\{4,5\}$. It suffices to consider the cases $e = e_{123}$ and $e = e_{234}$. For $e = e_{123}$, the vertex sequence $1, 2, 3, 4, 5$ and the hyperedge sequence $e, e_{1234}, e_{345}, e_{145}, e_{1235}$ form a Berge-$C_5^3$; for $e = e_{234}$, the same vertex sequence and the hyperedge sequence $e_{1234}, e, e_{345}, e_{145}, e_{1235}$ form a Berge-$C_5^3$.

Finally, we assume that $H = H_5''$ or $H = H_5'''$. For $H = H_5''$, $e$ contains at least one vertex from $\{1,2\}$; for $H = H_5'''$, $e$ contains at least two vertices from $\{2,3,4,5\}$. In both cases, it suffices to consider the cases $e = e_{123}$ and $e = e_{234}$. Thus the vertex sequence $1, 2, 3, 4, 5$ and the hyperedge sequence $e_{123}, e_{1234}, e', e_{1245}, e_{1235}$ or $e_{1234}, e_{234}, e', e_{1245}, e_{1235}$ form a Berge-$C_5^3$, where $e' = e_{345}$ for $H = H_5''$ and $e' = e_{1345}$ for $H = H_5'''$.
\end{proof}

\begin{lemma}\label{L12}
Let $H$ be a $\{3, 4\}$-graph on $n \geq 5$ vertices. If $\delta_3(H) \geq 1$, then $H$ contains a Berge-$C_5^3$ unless $H \in \mathcal{H}_5 \cup \{H_6, H_7, H_8\}$.
\end{lemma}

\begin{proof}
Let $V(H) = \{v_i: i \in [n]\}$. We assume that $H$ is Berge-$C_5^3$-free. We will show that $H \in \mathcal{H}_5 \cup \{H_6, H_7, H_8\}$.

\begin{case}\label{L12case1}
$n = 5$.
\end{case}
If every hyperedge of $H$ is of cardinality $3$ or $H$ has five $4$-hyperedges, then $H$ contains a Berge-$C_5^3$. Therefore the number of $4$-hyperedges in $H$ is between $1$ and $4$. Since every triple of $V(H)$ is contained in at least one hyperedge, we can see that $H$ must contain some member of $\mathcal{H}_5$ as a subhypergraph. From Lemma \ref{L13} and the fact that $H$ is Berge-$C_5^3$-free, we derive $H \in \mathcal{H}_5$.

\begin{case}\label{L12case2}
$n = 6$.
\end{case}
We define $\mathcal{H} = \{H_{V(H) \setminus \{v_i\}}^3: i \in [6]\}$. Since $H$ is Berge-$C_5^3$-free, Proposition \ref{P} implies that each hypergraph in $\mathcal{H}$ is Berge-$C_5^3$-free. By Case \ref{L12case1}, each hypergraph $H' \in \mathcal{H}$ is isomorphic to $H_5$, $H_5'$, $H_5''$ or $H_5'''$.

Suppose that there exists a hypergraph $H' \in \mathcal{H}$ isomorphic to $H_5'$, $H_5''$ or $H_5'''$. We assume $H' = H_{V(H) \setminus \{v_6\}}^3$ with the isomorphism mapping $v_i \mapsto i$ for all $i \in [5]$. Since $\delta_3(H) \geq 1$, $H$ contains hyperedges $h, h', h''$ such that $\{v_1, v_2, v_6\} \subseteq h$, $\{v_2, v_4, v_6\} \subseteq h'$ and $\{v_3, v_4, v_6\} \subseteq h''$. Since each $3$-hyperedge of $H'$ contains $v_5$, every $4$-hyperedge of $H$ containing $v_6$ must also contain $v_5$. We further have $\{v_3, v_4\} \cap h = \emptyset$, $\{v_1, v_3\} \cap h' = \emptyset$ and $\{v_1, v_2\} \cap h'' = \emptyset$. Consequently, the vertex sequence $v_1, v_2, v_6, v_4, v_3$ and the hyperedge sequence $h, h', h'', \{v_1, v_2, v_3, v_4\}, \{v_1, v_2, v_3, v_5\}$ form a Berge-$C_5^3$, a contradiction. We conclude that $H_{V(H) \setminus \{v_i\}}^3 \cong H_5$ for every $i \in [6]$. Since $\delta_3(H) \geq 1$ and the intersection of any two hyperedges in $H_5$ is of size at most two, every triple of $V(H)$ is contained in exactly one hyperedge of $H$.

Since every $3$-trace with five vertices of $H$ is isomorphic to $H_5$, we assume $H_0 = H^3_{V(H)\setminus\{v_6\}}$ with the isomorphism mapping $v_i\mapsto i$ for all $i\in[5]$. We can see that each hypergraph in $\mathcal{H}$ consists of six $3$-hyperedges and one $4$-hyperedge. The six hypergraphs in $\mathcal{H}$ totally yield thirty-six $3$-hyperedges and six $4$-hyperedges. Note that each $3$-hyperedge of $H$ contributes a $3$-hyperedge to three hypergraphs in $\mathcal{H}$, while each $4$-hyperedge in $H$ contributes a $3$-hyperedge to four hypergraphs and a $4$-hyperedge to two hypergraphs in $\mathcal{H}$. Denote by $x$ the number of $3$-hyperedges and $y$ the number of $4$-hyperedges in $H$. We obtain $3x + 4y = 36$ and $2y = 6$, which solves to $x = 8$ and $y = 3$. Hence, $H$ has eight $3$-hyperedges and three $4$-hyperedges.

Note that each $4$-hyperedge of $H$ is reduced to a $3$-hyperedge of $H_0$ when it contains $v_6$, and is still a $4$-hyperedge of $H_0$ otherwise. We conclude that $v_6$ is contained in two $4$-hyperedges and four $3$-hyperedges of $H$. Recall that $H$ has three $4$-hyperedges and one of them is $\{v_1,v_2,v_3,v_4\}$. The other two $4$-hyperedges of $H$, say $e$ and $e'$, must be extensions of two $3$-hyperedges of $H_0$ by adding the vertex $v_6$. This implies that $v_5$ is also contained in both $e$ and $e'$. Since every triple of $V(H)$ is contained in exactly one hyperedge of $H$, we have $e\cap e'=\{v_5,v_6\}$. Without loss of generality, assume that $e=\{v_1,v_2,v_5,v_6\}$ and $e'=\{v_3,v_4,v_5,v_6\}$. Then the 3-hyperedges of $H$ must be $\{v_1, v_3, v_5\}$, $\{v_1, v_4, v_5\}$, $\{v_2, v_3, v_5\}$, $\{v_2, v_4, v_5\}$, $\{v_1, v_3, v_6\}$, $\{v_1, v_4, v_6\}$, $\{v_2, v_3, v_6\}$ and $\{v_2, v_4, v_6\}$. This verifies $H \cong H_6$.

\begin{case}\label{L12case3}
$n = 7$.
\end{case}
We define $\mathcal{H} = \{H_{V(H) \setminus \{v_i\}}^3: i \in [7]\}$. Since $H$ is Berge-$C_5^3$-free, Proposition \ref{P} implies that each hypergraph in $\mathcal{H}$ is Berge-$C_5^3$-free. By Case \ref{L12case2}, each hypergraph in $\mathcal{H}$ is isomorphic to $H_6$. We assume $H_0 = H^3_{V(H)\setminus\{v_7\}}$ with the isomorphism mapping $v_i\mapsto i$ for all $i \in [6]$. By a similar analysis to Case \ref{L12case2}, the $3$-trace of any five vertices in $H$ can only be $H_5$. Therefore, every triple in $V(H)$ is contained in exactly one hyperedge of $H$. We can see that each hypergraph in $\mathcal{H}$ consists of eight $3$-hyperedges and three $4$-hyperedges. The seven hypergraphs in $\mathcal{H}$ totally yield fifty-six $3$-hyperedges and twenty-one $4$-hyperedges. Note that each $3$-hyperedge in $H$ contributes a $3$-hyperedge to four hypergraphs in $\mathcal{H}$, while each $4$-hyperedge in $H$ contributes a $3$-hyperedge to four hypergraphs and a $4$-hyperedge to three hypergraphs in $\mathcal{H}$. Denote by $x$ the number of $3$-hyperedges and $y$ the number of $4$-hyperedges in $H$. We obtain $4x + 4y = 56$ and $3y = 21$, which solves to $x = y = 7$. Hence, $H$ has seven $3$-hyperedges and seven $4$-hyperedges.

Note that each $4$-hyperedge of $H$ is reduced to a $3$-hyperedge of $H_0$ when it contains $v_7$, and is still a $4$-hyperedge of $H_0$ otherwise. We conclude that $v_7$ is contained in four $4$-hyperedges and three $3$-hyperedges of $H$. Without loss of generality, we assume that $\{v_1, v_3, v_5, v_7\} \in E(H)$. Since every triple of $V(H)$ is contained in exactly one hyperedge of $H$, we deduce that $\{v_1, v_3, v_6\}$, $\{v_1, v_4, v_5\}$ and $\{v_2, v_3, v_5\}$ are $3$-hyperedges in $H$. If $\{v_2, v_4, v_6, v_7\} \in E(H)$, then $\{v_2, v_4, v_5\}$, $\{v_2, v_3, v_6\}$ and $\{v_1, v_4, v_6\}$ are $3$-hyperedges in $H$, which implies at most two $4$-hyperedge contains $v_7$. It follows that $\{v_2, v_4, v_6\}$ is a $3$-hyperedge in $H$, and $\{v_2, v_4, v_5, v_7\}$, $\{v_2, v_3, v_6, v_7\}$, and $\{v_1, v_4, v_6, v_7\}$ are $4$-hyperedges in $H$. Since every triple of $V(H)$ is contained in exactly one hyperedge of $H$, $\{v_1, v_2, v_7\}$, $\{v_3, v_4, v_7\}$ and $\{v_5, v_6, v_7\}$ are $3$-hyperedges in $H$. This completes the proof that $H \cong H_7$.

\begin{case}\label{L12case4}
$n = 8$.
\end{case}
We define $\mathcal{H} = \{H_{V(H) \setminus \{v_i\}}^3: i \in [8]\}$. Since $H$ is Berge-$C_5^3$-free, Proposition \ref{P} implies that each hypergraph in $\mathcal{H}$ is Berge-$C_5^3$-free. By Case \ref{L12case3}, each hypergraph in $\mathcal{H}$ is isomorphic to $H_7$. We assume $H_0 = H^3_{V(H)\setminus\{v_8\}}$ with the isomorphism mapping $v_i\mapsto i$ for all $i \in [7]$. By a similar analysis to Case \ref{L12case2}, the $3$-trace of any five vertices in $H$ can only be $H_5$. Therefore, every triple in $V(H)$ is contained in exactly one hyperedge of $H$. We can see that each hypergraph in $\mathcal{H}$ consists of seven $3$-hyperedges and seven $4$-hyperedges. The eight hypergraphs in $\mathcal{H}$ totally yield fifty-six $3$-hyperedges and fifty-six $4$-hyperedges. Note that each $3$-hyperedge in $H$ contributes a $3$-hyperedge to five hypergraphs in $\mathcal{H}$, while each $4$-hyperedge in $H$ contributes a $3$-hyperedge to four hypergraphs and a $4$-hyperedge to four hypergraphs in $\mathcal{H}$. Denote by $x$ the number of $3$-hyperedges and $y$ the number of $4$-hyperedges in $H$. We obtain $5x + 4y = 56$ and $4y = 56$, which solves to $x = 0$ and $y = 14$. Hence, $H$ has fourteen $4$-hyperedges.

Note that each $4$-hyperedge of $H$ is reduced to a $3$-hyperedge of $H_0$ when it contains $v_8$, and is still a $4$-hyperedge of $H_0$ otherwise. We conclude that $v_8$ is contained in four $4$-hyperedges and three $3$-hyperedges of $H$. Thus, all hyperedges containing $v_8$ are constructed by adding $v_8$ to some $3$-hyperedge of $H_7$, which completes the proof that $H \cong H_8$.

\begin{case}\label{L12case5}
$n \geq 9$.
\end{case}
Since $H$ is Berge-$C_5^3$-free, Proposition \ref{P} implies that $H_{\{v_i: i \in [8]\}}^3$ and $H_{\{v_i: i \in [2,9]\}}^3$ is Berge-$C_5^3$-free. By Case \ref{L12case4}, both of these two hypergraphs are isomorphic to $H_8$. However, every $4$-hyperedge containing $v_1$ in $H_{\{v_i: i \in [8]\}}^3$ serves as a $3$-hyperedge in $H_{\{v_i: i \in [2,9]\}}^3$. This contradicts that $H_8$ has no $3$-hyperedges.
\end{proof}

\begin{lemma}\label{L}
Let $H$ be a $\{3, 4\}$-graph on $n \geq 6$ vertices with $\delta_3(H) \geq 1$. If $H$ contains a Berge-$C_t^3$ for some $5 \leq t \leq n-1$, then $H$ contains a Berge-$C_{t+1}^3$.
\end{lemma}
\begin{proof}
Let $V(H) = \{v_i: i \in [n]\}$, and let $C_0$ be a Berge-$C_t^3$ in $H$ with vertex sequence $v_1, v_2, \dots, v_t$ and hyperedge sequence $h_1, h_2, \dots, h_t$. We write $T = \{v_i: i \in [t]\}$. We claim that $H_T^3$ also contains a Berge-$C_t^3$ with the identical vertex sequence $v_1, v_2, \dots, v_t$ and the hyperedge sequence $h_1', h_2', \dots, h_t'$, where $h_i' = h_i \cap T$ for each $i \in [t]$. Suppose to the contrary that $h_i' = h_j'$ for some distinct $i, j \in [t]$. Since $\{v_i,v_{i+1},v_{i+2}\}\subseteq h_i'$, $\{v_j,v_{j+1},v_{j+2}\}\subseteq h_j'$ and $\{v_i,v_{i+1},v_{i+2}\}\neq \{v_j,v_{j+1},v_{j+2}\}$, we have $|h_i'| = |h_j'| \geq 4$. Note that $h_i' \subseteq h_i$ and $h_j' \subseteq h_j$. We get $|h_i| = |h_j| = 4$ and consequently $h_i = h_j$, a contradiction. By Proposition \ref{P}, for every $v \in V(H) \setminus T$, $H_{T \cup \{v\}}^3$ contains a Berge-$C_t^3$ with the same vertex sequence $v_1, v_2, \dots, v_t$.

Given a vertex cyclic ordering $S$, we use $V(S)$ to denote the set of vertices in $S$. For a cyclic ordering $S$ of $t$ vertices in $H$ and a vertex $u \in V(H) \setminus V(S)$, we call the pair $(S, u)$ \emph{admissible} with respect to $H$ if $H_{V(S) \cup \{u\}}^3$ contains a Berge-$C_t^3$ with the vertex sequence $S$. For each admissible pair $(S, u)$, say $S = (v_1, v_2, \dots, v_t)$, define
$$\sigma(S, u) = |\{i: \{v_{i-1}, v_i, v_{i+1}, u\} \in E(H)\}|,$$
with indices modulo $t$. The analysis in the previous paragraph implies the existence of an admissible pair in $H$.

Let $(S_0, v)$ be an admissible pair with $\sigma(S_0, v)$ as large as possible. Without loss of generality, we take $S_0 = (v_1, v_2, \dots, v_t)$. Let $D = v_1v_2 \dots v_tv_1$ denote a cycle (which is a $2$-graph). We define a $2$-coloring on $V(D)$ by coloring $v_i$ red if $\{v_{i-1}, v_i, v_{i+1}, v\} \in E(H)$, and blue otherwise. Let $H_0 = H_{V(S_0) \cup \{v\}}^3$. Since $(S_0, v)$ is admissible, $H_0$ contains a Berge-$C_t^3$ with the vertex sequence $v_1, v_2, \dots, v_t$ and the hyperedge sequence $e_1, e_2, \dots, e_t$. Note that the hyperedge $\{v_{i-1}, v_i, v_{i+1}, v\}$ can only serve as $e_{i-1}$. Accordingly, we assume $e_{i-1} = \{v_{i-1}, v_i, v_{i+1}, v\}$ whenever $v_i$ is red in $D$. From now on, we denote this Berge-$C_t^3$ by $C$. A hyperedge $h \in E(H_0)$ is called \emph{unpicked} if $h$ is not a hyperedge of $C$.

By Proposition \ref{P}, we have $\delta_3(H_0) \geq 1$. It suffices to find a Berge-$C_{t+1}^3$ in $H_0$ to guarantee the existence of a Berge-$C_{t+1}^3$ in $H$. For any vertex sequence $v_1', v_2', \dots, v_{t+1}'$ of $V(H_0)$, there exists a hyperedge sequence $e_1', e_2', \dots, e_{t+1}'$ such that $\{v_i', v_{i+1}', v_{i+2}'\} \subseteq e_i'$ for each $i \in [t+1]$ (the indices are taken modulo $t + 1$). This sequence may contain repeated hyperedges, which does not form a Berge-$C_{t+1}^3$. Since $H_0$ is a $\{3, 4\}$-graph, $e_i' = e_j'$ ($i \neq j$) implies $j = i \pm 1$. In the following, we will find a Berge-$C_{t+1}^3$ in $H_0$ based on $C$.

The proof of Lemma \ref{L} is given by distinguishing the numbers of red and blue vertices in $D$. We first deal with the case where one of the two numbers is large. After this process, only a few cases are left, and the structures are extremely clear. Hence, we analyse all possible configurations to complete the proof of Lemma \ref{L}.

\setcounter{claim}{0}
\begin{claim}\label{LC3}
If $D$ has two consecutive blue vertices $v_i$ and $v_{i+1}$ for some $i \in [t]$, then $H_0$ contains a Berge-$C_{t+1}^3$.
\end{claim}
\begin{proof}
We first prove for the case that $D$ has four consecutive blue vertices $v_i$, $v_{i+1}$, $v_{i+2}$, $v_{i+3}$. Since $\delta_3(H_0) \geq 1$, there exist unpicked hyperedges $h_1$, $h_2$, $h_3$ such that $\{v_i, v_{i+1}, v\} \subseteq h_1$, $\{v_{i+1}, v_{i+2}, v\} \subseteq h_2$ and $\{v_{i+2}, v_{i+3}, v\} \subseteq h_3$. We have $h_1 \neq h_2$. Otherwise there exists a hyperedge $h = \{v_i, v_{i+1}, v_{i+2}, v\}$, contradicting the assumption that $v_{i+1}$ is a blue vertex in $D$. Similarly, we have $h_2 \neq h_3$. Now the vertex sequence $v_{i-1}, v_i, v_{i+1}, v, v_{i+2}, v_{i+3}, \dots$ and the hyperedge sequence $e_{i-1}, h_1, h_2, h_3, e_{i+2}, e_{i+3}, \dots$ form a Berge-$C_{t+1}^3$.

Next, we prove for the case that $D$ has five consecutive vertices $v_{i-1}$, $v_i$, $v_{i+1}$, $v_{i+2}$, $v_{i+3}$ with $v_i$, $v_{i+1}$, $v_{i+2}$ being blue and $v_{i-1}$, $v_{i+3}$ being red. Then $e_{i-2} = \{v_{i-2}, v_{i-1}, v_i, v\}$ and $e_{i+2} = \{v_{i+2}, v_{i+3}, v_{i+4}, v\}$. Since $\delta_3(H_0) \geq 1$, there exist two unpicked hyperedges $h_1$, $h_2$ such that $\{v_i, v_{i+1}, v\} \subseteq h_1$ and $\{v_{i+1}, v_{i+2}, v\} \subseteq h_2$. We have $h_1 \neq h_2$. Otherwise there exists a hyperedge $h = \{v_i, v_{i+1}, v_{i+2}, v\}$, contradicting the assumption that $v_{i+1}$ is a blue vertex in $D$.
Since $\delta_3(H_0) \geq 1$, there are unpicked hyperedges $h_3$, $h_4$ such that $\{v_{i-1}, v_{i+1}, v\} \subseteq h_3$ and $\{v_{i+1}, v_{i+3}, v\} \subseteq h_4$. We have $h_3 \neq h_1$ and $h_4 \neq h_2$. Otherwise there exists a hyperedge $h_3 = \{v_{i-1}, v_i, v_{i+1}, v\}$ or $h_4 = \{v_{i+1}, v_{i+2}, v_{i+3}, v\}$, contradicting the assumption that $v_i$ or $v_{i+2}$ is a blue vertex in $D$.
Since $\delta_3(H_0) \geq 1$, there exist hyperedges $h_5$, $h_6$ such that $\{v_i, v_{i+2}, v_{i+3}\} \subseteq h_5$ and $\{v_{i-1}, v_i, v_{i+2}\} \subseteq h_6$. If $e_i \neq h_5$, then the vertex sequence $v_{i-2}, v_{i-1}, v, v_{i+1}, v_i, v_{i+2}, v_{i+3}, \dots$ and the hyperedge sequence $e_{i-2}, h_3, h_1, e_i, h_5, e_{i+2}, e_{i+3}, \dots$ form a Berge-$C_{t+1}^3$. Thus $e_i = h_5 = \{v_i, v_{i+1}, v_{i+2}, v_{i+3}\}$. Now the vertex sequence $v_{i-2}, v_{i-1}, v_i, v_{i+2}, v_{i+1}, v, v_{i+3}, \dots$ and the hyperedge sequence $e_{i-2}, h_6, e_i, h_2, h_4, e_{i+2}, e_{i+3}, \dots$ form a Berge-$C_{t+1}^3$.

The remaining case is that $v_i$, $v_{i+1}$ are blue and $v_{i-1}$, $v_{i+2}$ are red. Then $e_{i-2} = \{v_{i-2}, v_{i-1},$ $v_i, v\}$ and $e_{i+1} = \{v_{i+1}, v_{i+2}, v_{i+3}, v\}$. Since $\delta_3(H_0) \geq 1$, there are unpicked hyperedges $h_1$, $h_2$, $h_3$ such that $\{v_i, v_{i+1}, v\} \subseteq h_1$, $\{v_{i-1}, v_{i+1}, v\} \subseteq h_2$ and $\{v_i, v_{i+2}, v\} \subseteq h_3$. We have $h_1 \neq h_2$ and $h_1 \neq h_3$. Otherwise there exists a hyperedge $h = \{v_{i-1}, v_i, v_{i+1}, v\}$ or $h = \{v_i, v_{i+1}, v_{i+2}, v\}$, contradicting the assumption that $v_{i+1}$ is a blue vertex in $D$.
Note that there are hyperedges $h_4$, $h_5$ with $\{v_i, v_{i+2}, v_{i+3}\} \subseteq h_4$ and $\{v_{i-2}, v_{i-1}, v_{i+1}\} \subseteq h_5$ since $\delta_3(H_0) \geq 1$. This forces that $h_4 = e_i = \{v_i, v_{i+1}, v_{i+2}, v_{i+3}\}$ or $h_4 = e_{i+2} = \{v_i, v_{i+2}, v_{i+3}, v_{i+4}\}$. Otherwise the vertex sequence $v_{i-2}, v_{i-1}, v, v_{i+1}, v_i, v_{i+2}, v_{i+3}, \dots$ and the hyperedge sequence $e_{i-2}, h_2, h_1, e_i, h_4, e_{i+2}, e_{i+3}, \dots$ form a Berge-$C_{t+1}^3$. Similarly, we get $h_5 = e_{i-3} = \{v_{i-3}, v_{i-2}, v_{i-1}, v_{i+1}\}$ or $h_5 = e_{i-1} = \{v_{i-2}, v_{i-1}, v_i, v_{i+1}\}$.
If $h_4 = e_i$ and $h_5 = e_{i-1}$, then the vertex sequence $v_{i-3}, v_{i-2}, v_{i-1}, v_i, v, v_{i+1},$ $v_{i+2}, v_{i+3}, \dots$ and the hyperedge sequence $e_{i-3}, e_{i-1}, e_{i-2}, h_1, e_{i+1}, e_i, e_{i+2}, e_{i+3}, \dots$ form a Berge-$C_{t+1}^3$.
If $h_4 = e_{i+2}$ and $h_5 = e_{i-3}$, then $t \geq 6$ because $h_4 \neq h_5$. Since $\delta_3(H_0) \geq 1$, there exist hyperedges $h_6$, $h_6'$, $h_6''$ such that $\{v_{i-2}, v_{i+1}, v_{i+2}\} \subseteq h_6$, $\{v_{i-1}, v_{i+2}, v\} \subseteq h_6'$ and $\{v_{i-1}, v_i, v_{i+3}\} \subseteq h_6''$. Now the vertex sequence $v_{i-3}, v_{i-2}, v_{i+1}, v_{i+2}, v, v_{i-1}, v_i, v_{i+3}, v_{i+4}, \dots$ and the hyperedge sequence $e_{i-3}, h_6, e_{i+1}, h_6', e_{i-2}, h_6'', e_{i+2}, e_{i+3}, e_{i+4}, \dots$ form a Berge-$C_{t+1}^3$.
If $h_4 = e_i$ and $h_5 = e_{i-3}$, then $\delta_3(H_0) \geq 1$ implies that the existence of hyperedges $h_7$ and $h_7'$ such that $\{v_{i-2}, v_{i+1}, v\} \subseteq h_7$ and $\{v_{i-1}, v_i, v_{i+2}\} \subseteq h_7'$. We have $h_7 = h_2 = \{v_{i-2}, v_{i-1}, v_{i+1}, v\}$. Otherwise the vertex sequence $v_{i-3}, v_{i-2}, v_{i+1}, v, v_{i-1}, v_i, v_{i+2}, v_{i+3}, \dots$ and the hyperedge sequence $e_{i-3}, h_7, h_2, e_{i-2}, h_7', e_i, e_{i+2}, e_{i+3}, \dots$ form a Berge-$C_{t+1}^3$. It follows that the vertex sequence $v_{i-3}, v_{i-2}, v_{i-1}, v, v_i, v_{i+1}, \dots$ and the hyperedge sequence $e_{i-3}, h_2, e_{i-2}, h_1, e_i, e_{i+1}, \dots$ form a Berge-$C_{t+1}^3$. The case where $h_4 = e_{i+2}$ and $h_5 = e_{i-1}$ can be verified similarly because of the symmetry.
\end{proof}

\begin{claim}\label{LC4}
If $D$ has four red vertices $v_i$, $v_{i+2}$, $v_j$, $v_{j+2}$ with $j \notin \{i-3, i-2, i-1, i, i+1, i+2\}$, then $H_0$ contains a Berge-$C_{t+1}^3$.
\end{claim}
\begin{proof}
Since $\delta_3(H_0) \geq 1$, there exist hyperedges $h_0$, $h_0'$, $h_0''$, $h$, $h'$ such that $\{v_i, v_j, v\} \subseteq h_0$, $\{v_{i-1}, v_j, v\} \subseteq h_0'$, $\{v_{i-2}, v_j, v\} \subseteq h_0''$, $\{v_{i+1}, v_{i+2}, v_{j+1}\} \subseteq h$ and $\{v_{i+1}, v_{j+1}, v_{j+2}\} \subseteq h'$.
We have $h = h' = \{v_{i+1}, v_{i+2}, v_{j+1}, v_{j+2}\}$; otherwise the vertex sequence $v_{i-1}, v_i, v, v_j, v_{j-1}, \dots, v_{i+3}, v_{i+2},$ $v_{i+1}, v_{j+1}, v_{j+2}, \dots$ and the hyperedge sequence $e_{i-1}, h_0, e_{j-1}, e_{j-2}, e_{j-3}, \dots$, $e_{i+1}, h, h', e_{j+1},$ $e_{j+2}, \dots$ form a Berge-$C_{t+1}^3$.
Furthermore, we have $e_i \neq \{v_i, v_{i+1}, v_{i+2}, v_{j+1}\}$; otherwise the vertex sequence $v_{i-1}, v_i, v, v_j, v_{j-1}, \dots, v_{i+3}, v_{i+2}, v_{i+1}, v_{j+1}, v_{j+2}, \dots$ and the hyperedge sequence $e_{i-1}, h_0, e_{j-1}, e_{j-2}, e_{j-3}, \dots$, $e_{i+1}, e_i, h, e_{j+1}, e_{j+2}, \dots$ form a Berge-$C_{t+1}^3$.

\setcounter{case}{0}
\begin{case}
$e_{i-2} = \{v_{i-2}, v_{i-1}, v_i, v\}$.
\end{case}
Since $\delta_3(H_0) \geq 1$, there exist hyperedges $h_1$, $h_1'$ such that $\{v_i, v_{i+1}, v_{j+1}\} \subseteq h_1$ and $\{v_i, v_{j+1},$ $v_{j+2}\} \subseteq h_1'$. We get $h_1 = h_1' = \{v_i, v_{i+1}, v_{j+1}, v_{j+2}\}$, since otherwise the vertex sequence $v_{i-2}, v_{i-1}, v, v_j, v_{j-1}, \dots, v_{i+2}, v_{i+1}, v_i, v_{j+1}, v_{j+2}, \dots$ and the hyperedge sequence $e_{i-2}, h_0', e_{j-1},$ $e_{j-2}, e_{j-3}, \dots, e_i, h_1, h_1', e_{j+1}, e_{j+2}, \dots$ form a Berge-$C_{t+1}^3$. Thus the vertex sequence $v_{i-1}, v_i, v,$ $v_j, v_{j-1}, \dots, v_{i+3}, v_{i+2}, v_{i+1}, v_{j+1}, v_{j+2}, \dots$ and the hyperedge sequence $e_{i-1}, h_0, e_{j-1}, e_{j-2}, e_{j-3},$ $\dots, e_{i+1}, h, h_1, e_{j+1}, e_{j+2}, \dots$ form a Berge-$C_{t+1}^3$.

\begin{case}
$e_{i-2} \neq \{v_{i-2}, v_{i-1}, v_i, v\}$.
\end{case}
Note that $v_{i-1}$ is a blue vertex in this case. By Claim \ref{LC3}, the vertex $v_{i-2}$ is a red and $e_{i-3} = \{v_{i-3}, v_{i-2}, v_{i-1}, v\}$. Since $\delta_3(H_0) \geq 1$, there exist hyperedges $h_2$, $h_2'$, $h_2''$, $h_2'''$, $h_2''''$ such that $\{v_{i-1}, v_i, v_{i+2}\} \subseteq h_2$, $\{v_{i-1}, v_{i+2}, v_{i+3}\} \subseteq h_2'$, $\{v_i, v_{i+1}, v_{j+1}\} \subseteq h_2''$, $\{v_i, v_{i+2}, v_{i+3}\} \subseteq h_2'''$ and $\{v_{i-1}, v_{i+1}, v_{j+1}\} \subseteq h_2''''$.

We have $h_2' = h_2$ or $h_2' = e_{i+2}$; otherwise the vertex sequence $v_{i-3}, v_{i-2}, v, v_j, v_{j-1}, \dots, v_{i+4},$ $v_{i+3}, v_{i+2}, v_{i-1}, v_i, v_{i+1}, v_{j+1}, v_{j+2}, \dots$ and the hyperedge sequence $e_{i-3}, h_0'', e_{j-1}, e_{j-2}, e_{j-3}, \dots$, $e_{i+2}, h_2', h_2, e_{i-1}, h_2'', h, e_{j+1}, e_{j+2}, \dots$ form a Berge-$C_{t+1}^3$.
We also have $h_2''' = h_2$ or $h_2''' = e_{i+2}$; otherwise the vertex sequence $v_{i-3}, v_{i-2}, v, v_j, v_{j-1}, \dots, v_{i+4}, v_{i+3}, v_{i+2}, v_i, v_{i-1}, v_{i+1}, v_{j+1}, v_{j+2},$ $\dots$ and the hyperedge sequence $e_{i-3}, h_0'', e_{j-1}, e_{j-2}, e_{j-3}, \dots$, $e_{i+2}, h_2''', h_2, e_{i-1}, h_2'''', h, e_{j+1}, e_{j+2},$ $\dots$ form a Berge-$C_{t+1}^3$.
Thus we conclude that $h_2 = h_2'$ or $h_2 = h_2'''$, which implies $h_2 = \{v_{i-1}, v_i, v_{i+2}, v_{i+3}\}$.

\begin{subcase}
$j = i-4$.
\end{subcase}
Since $\delta_3(H_0) \geq 1$, there exist hyperedges $h_3$, $h_3'$ such that $\{v_{i-2}, v_i, v\} \subseteq h_3$ and $\{v_{i-3}, v_{i-1},$ $v_{i+1}\} \subseteq h_3'$. We have $h_3 = h_0 = \{v_{i-4}, v_{i-2}, v_i, v\}$; otherwise the vertex sequence $v_{i-5}, v_{i-4}, v, v_i,$ $v_{i-2}, v_{i-1}, v_{i-3}, v_{i+1}, v_{i+2}, \dots$ and the hyperedge sequence $e_{i-5}, h_0, h_3, e_{i-2}, e_{i-3}, h_3', h, e_{i+1}, e_{i+2},$ $\dots$ form a Berge-$C_{t+1}^3$.
Since $\delta_3(H_0) \geq 1$, there exist hyperedges $h_4$, $h_4'$, $h_4''$ such that $\{v_{i-3}, v_i, v\}$ $\subseteq h_4$, $\{v_{i-3}, v_{i-1}, v_i\} \subseteq h_4'$ and $\{v_{i-2}, v_{i-1}, v_{i+1}\} \subseteq h_4''$. We have $h_4 = h_4' = \{v_{i-3}, v_{i-1}, v_i, v\}$; otherwise the vertex sequence $v_{i-5}, v_{i-4}, v, v_i, v_{i-3}, v_{i-1}, v_{i-2}, v_{i+1}, v_{i+2}, \dots$ and the hyperedge sequence $e_{i-5}, h_3, h_4, h_4', e_{i-3}, h_4'', h, e_{i+1}, e_{i+2}, \dots$ form a Berge-$C_{t+1}^3$.
Since $\delta_3(H_0) \geq 1$, there exist hyperedges $h_5$, $h_5'$ such that $\{v_{i-4}, v_{i-3}, v_{i-1}\} \subseteq h_5$ and $\{v_{i-2}, v_{i+1}, v\} \subseteq h_5'$. Thus the vertex sequence $v_{i-5}, v_{i-4}, v_{i-3}, v_{i-1}, v_i, v, v_{i-2}, v_{i+1}, v_{i+2}, \dots$ and the hyperedge sequence $e_{i-5}, h_5, h_4,$ $e_{i-1}, h_3, h_5', h, e_{i+1}, e_{i+2}, \dots$ form a Berge-$C_{t+1}^3$.

\begin{subcase}
$j \neq i-4$.
\end{subcase}
Since $\delta_3(H_0) \geq 1$, there exist hyperedges $h_6, h_6'$ such that $\{v_{i-1}, v_{i+1}, v_{j+1}\} \subseteq h_6$ and $\{v_{i-1},$ $v_{j+1}, v_{j+2}\} \subseteq h_6'$. We have $h_6 = h_6' = \{v_{i-1}, v_{i+1}, v_{j+1}, v_{j+2}\}$, since otherwise the vertex sequence $v_{i-3}, v_{i-2}, v, v_j, v_{j-1}, \dots, v_{i+4}, v_{i+3}, v_{i+2}, v_i, v_{i+1}, v_{i-1}, v_{j+1}, v_{j+2}, \dots$ and the hyperedge sequence $e_{i-3}, h_0'', e_{j-1}, e_{j-2}, e_{j-3}, \dots$, $e_{i+2}, h_2, e_i, e_{i-1}, h_6, h_6', e_{j+1}, e_{j+2}, \dots$ form a Berge-$C_{t+1}^3$. The vertex sequence $v_{i-1}, v_i, v, v_j, v_{j-1}, \dots, v_{i+3}, v_{i+2}, v_{i+1}, v_{j+1}, v_{j+2}, \dots$ and the hyperedge sequence $e_{i-1}, h_0, e_{j-1}, e_{j-2}, e_{j-3}, \dots, e_{i+1}, h, h_6, e_{j+1}, e_{j+2}, \dots$ form a Berge-$C_{t+1}^3$.
\end{proof}

\begin{claim}\label{LC5}
If $t \geq 7$, then $H_0$ contains a Berge-$C_{t+1}^3$.
\end{claim}
\begin{proof}
If all vertices are red, then $H_0$ contains a Berge-$C_{t+1}^3$ by Claim \ref{LC4}. We may assume that $v_1$ is blue. 
Suppose to the contrary that $H_0$ is Berge-$C_{t+1}^3$-free.
By Claim \ref{LC3}, $v_2$ and $v_t$ must be red. If $v_i$ is blue for some $i$ with $4 \leq i \leq t-2$, then $v_{i-1}$ and $v_{i+1}$ are red by Claim \ref{LC3}. Apply Claim \ref{LC4} to $\{v_2, v_{i-1} ,v_{i+1}, v_t\}$, we obtain a Berge-$C_{t+1}^3$, a contradiction. Thus $v_i$ is red for all $4 \leq i \leq t-2$. In particular, vertices $v_2$, $v_4$, $v_{t-2}$ and $v_t$ are red. Then we get a Berge-$C_{t+1}^3$ by applying Claim \ref{LC4} to these four vertices, a contradiction. In summary, $H_0$ always contains a Berge-$C_{t+1}^3$ when $t\geq 7$.
\end{proof}

Now we may assume that $t = 5$ or $6$. By Claim \ref{LC3}, we further assume that the blue vertices form an independent set of $D$. Therefore, $D$ must be one of the following 8 configurations.

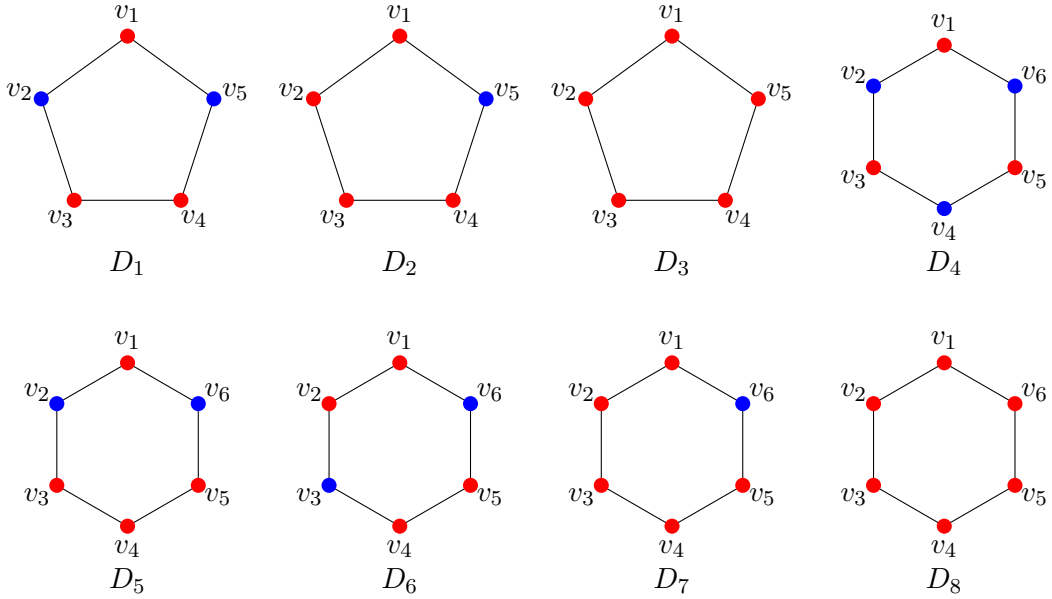
\begin{figure}[htbp]
  \centering
  \begin{tikzpicture}[scale=0.6]
    \drawCfive{0}{0}{$D_1$}{red}{blue}{red}{red}{blue};
    \drawCfive{6}{0}{$D_2$}{red}{red}{red}{red}{blue};
    \drawCfive{12}{0}{$D_3$}{red}{red}{red}{red}{red};
    \drawCsix{18}{0}{$D_4$}{red}{blue}{red}{blue}{red}{blue};
    \drawCsix{0}{-7}{$D_5$}{red}{blue}{red}{red}{red}{blue};
    \drawCsix{6}{-7}{$D_6$}{red}{red}{blue}{red}{red}{blue};
    \drawCsix{12}{-7}{$D_7$}{red}{red}{red}{red}{red}{blue};
    \drawCsix{18}{-7}{$D_8$}{red}{red}{red}{red}{red}{red};
  \end{tikzpicture}
  \caption{All possible configurations of $D$.}
  \label{F}
\end{figure}

We will extend $C$ to a Berge-$C_{t+1}^3$ in $H_0$ under all configurations.

\setcounter{case}{0}
\begin{case}
$D = D_1$.
\end{case}
By the coloring of $V(D)$, we know that $e_2 = \{v_2, v_3, v_4, v\}$, $e_3 = \{v_3, v_4, v_5, v\}$, $e_5 = \{v_1, v_2, v_5, v\}$ and $v \notin e_1 \cup e_4$.
Since $\delta_3(H_0) \geq 1$, there are hyperedges $h_1,h_2,\ldots,h_7$ with
$$\{v_1, v_3, v_4\} \subseteq h_1, \{v_1, v_4, v\} \subseteq h_2, \{v_2, v_3, v_5\} \subseteq h_3, \{v_1, v_3, v\} \subseteq h_4,$$
$$\{v_2, v_4, v_5\} \subseteq h_5, \{v_1, v_2, v_4\} \subseteq h_6, \{v_1, v_3, v_5\} \subseteq h_7.$$

Assume $h_1 = \{v_1, v_3, v_4, v_5\}$. If the vertex sequence $v_1, v, v_4, v_5, v_3, v_2$ and the hyperedge sequence $h_2, e_3, h_1, h_3, e_1, e_5$ do not form a Berge-$C_6^3$, then there are repeated edges, which can only be the situation that $h_3 = e_1 = \{v_1, v_2, v_3, v_5\}$. We further consider the the vertex sequence $v_1, v, v_3, v_2, v_5, v_4$ and the hyperedge sequence $h_4, e_2, e_1, h_5, h_1, h_2$. It is a Berge-$C_6^3$ unless $h_2 = h_4 = \{v_1, v_3, v_4, v\}$. However, in the latter case,  
looking at the vertex sequence $v_1, v_3, v, v_5, v_2, v_4$ and the hyperedge sequence $h_2, e_3, e_5, h_5, h_6, h_1$, we get a Berge-$C_6^3$ unless $h_5 = h_6 = \{v_1, v_2, v_4, v_5\}$. Now with the full knowledge of hyperedges $h_1,e_1,h_2,h_5$, we can see that the vertex sequence $v_1, v_2, v_5, v_3, v, v_4$ and the hyperedge sequence $e_5, e_1, e_3, e_2, h_2, h_5$ form a Berge-$C_6^3$. 
Similarly, when $h_1 =\{v_1, v_2, v_3,$ $v_4\}$, we also obtain a Berge-$C_6^3$. Thus we may assume that $h_1 = \{v_1, v_3, v_4\}$ or $h_1 = \{v_1, v_3, v_4, v\}$.  Now the vertex sequence $v_1, v_3, v_5, v, v_2, v_4$ and the hyperedge sequence $h_7, e_3, e_5, e_2, h_5, h_1$ form a Berge-$C_6^3$.

\begin{case}
$D \in \{D_2, D_3\}$.
\end{case}

In this case, vertices $v_1$, $v_2$, $v_3$, $v_4$ are red. So $e_1 = \{v_1, v_2, v_3, v\}$, $e_2 = \{v_2, v_3, v_4, v\}$, $e_3 = \{v_3, v_4, v_5, v\}$ and $e_5 = \{v_1, v_2, v_5, v\}$. Since $\delta_3(H_0) \geq 1$, there exist hyperedges $h_i$ ($1 \leq i \leq 5$) such that
$$\{v_1, v_3, v_4\} \subseteq h_1, \{v_1, v_4, v\} \subseteq h_2, \{v_2, v_3, v_5\} \subseteq h_3, \{v_1, v_2, v_4\} \subseteq h_4, \{v_1, v_3, v_5\} \subseteq h_5.$$
If $h_1 = \{v_1, v_2, v_3, v_4\}$, then the vertex sequence $v_1, v_2, v_3, v, v_4, v_5$ and the hyperedge sequence $h_1, e_1, e_2, e_3, e_4, e_5$ form a Berge-$C_6^3$. If $h_1 = \{v_1, v_3, v_4, v_5\}$, then the vertex sequence $v_1, v, v_4, v_5,$ $v_3, v_2$ and the hyperedge sequence $h_2, e_3, h_1, h_3, e_1, e_5$ form a Berge-$C_6^3$.
Thus we have $h_1 = \{v_1, v_3, v_4\}$ or $h_1 = \{v_1, v_3, v_4, v\}$. Then the vertex sequence $v_1, v_4, v_2, v, v_5, v_3$ and the hyperedge sequence $h_4, e_2, e_5, e_3, h_5, h_1$ form a Berge-$C_6^3$.

\begin{case}
$D \in \{D_4, D_5, D_7, D_8\}$.
\end{case}
In this case, $v_1$, $v_3$, $v_5$ are red vertices. We have $e_2=\{v_2, v_3, v_4, v\}$, $e_4=\{v_4, v_5, v_6, v\}$ and $e_6=\{v_1, v_2, v_6, v\}$. Since $\delta_3(H_0) \geq 1$, there exist hyperedges $h_i$ ($1 \leq i \leq 5$) such that
$$\{v_1, v_5, v\} \subseteq h_1, \{v_3, v_5, v_6\} \subseteq h_2, \{v_3, v_4, v_6\} \subseteq h_3, \{v_1, v_2, v_4\} \subseteq h_4, \{v_1, v_4, v\} \subseteq h_5.$$
We may assume that $h_2 = h_3 = \{v_3, v_4, v_5, v_6\}$; otherwise the vertex sequence $v_1, v, v_5, v_6, v_3, v_4, v_2$ and the hyperedge sequence $h_1, e_4, h_2, h_3, e_2, h_4, e_6$ form a Berge-$C_7^3$. By symmetry, there exist three hyperedges $h_6 = \{v_2, v_3, v_4, v_5\}$, $h_7 = \{v_1, v_2, v_3, v_6\}$ and $h_8 = \{v_1, v_2, v_5, v_6\}$. Thus the vertex sequence $v_1, v, v_4, v_5, v_3, v_2, v_6$ and the hyperedge sequence $h_5, e_4, h_2, h_6, h_7, h_8, e_6$ form a Berge-$C_7^3$.

\begin{case}
$D = D_6$.
\end{case}

Since $\delta_3(H_0) \geq 1$, there exist hyperedges $h_i$ ($1 \leq i \leq 6$) such that
$$\{v_1, v_4, v\} \subseteq h_1, \{v_3, v_5, v_6\} \subseteq h_2, \{v_2, v_3, v_6\} \subseteq h_3,$$
$$\{v_1, v_4, v_5\} \subseteq h_4, \{v_2, v_4, v\} \subseteq h_5, \{v_1, v_5, v\} \subseteq h_6.$$
We may assume that $h_2 = h_3 = \{v_2, v_3, v_5, v_6\}$; otherwise the vertex sequence $v_1, v, v_4, v_5, v_6, v_3, v_2$ and the hyperedge sequence $h_1, e_3, e_4, h_2, h_3, e_1, e_6$ form a Berge-$C_7^3$. By symmetry, there exists a hyperedge $h_7 = \{v_1, v_3, v_4, v_6\}$.
We may further assume that $h_4 = e_5 = \{v_1, v_4, v_5, v_6\}$; otherwise the vertex sequence $v_1, v_5, v_4, v, v_2, v_3, v_6$ and the hyperedge sequence $h_4, e_3, h_5, e_1, h_2, h_7, e_5$ form a Berge-$C_7^3$.
Thus the vertex sequence $v_1, v, v_5, v_6, v_4, v_3, v_2$ and the hyperedge sequence $h_6, e_4, e_5, h_7, e_2, e_1, e_6$ form a Berge-$C_7^3$.

Consequently, we can always find a Berge-$C_{t+1}^3$ in $H_0$. By Proposition \ref{P}, $H$ also contains a Berge-$C_{t+1}^3$. This completes the proof.
\end{proof}

\begin{proof}[Proof of Theorem \ref{T}]
By Lemmas \ref{L11} and \ref{L12}, the cases $n = 4$ and $n = 5$ are settled. For $n = 6$, Lemmas \ref{L11} and \ref{L12} implies that $H$ always contains a Berge-$C_4^3$ and it contains a Berge-$C_5^3$ unless $H = H_6$. By Lemma \ref{L}, we conclude that $H$ contains a Berge-$C_6^3$ if $H \neq H_6$. Furthermore, if $H = H_6$, then the vertex sequence $v_1, v_3, v_5, v_2, v_4, v_6$ and the corresponding $3$-hyperedges of $H_6$ form a $C_6^3$.

For $n \geq 7$, Lemmas \ref{L11} and \ref{L12} ensures that $H$ always contains a Berge-$C_4^3$ and it contains a Berge-$C_5^3$ unless $H \in \{H_7, H_8\}$. Moreover, if $H = H_7$ or $H = H_8$, then $H_S^3 \cong H_6$ for every $6$-vertex subset $S \subseteq V(H)$, so $H$ contains a Berge-$C_6^3$  by Proposition \ref{P}. Thus, by Lemma \ref{L}, we conclude that $H$ contains a Berge-$C_t^3$ for all $t \geq 6$. This completes the proof of Theorem \ref{T}.
\end{proof}


\begin{thebibliography}{28}
\bibitem{AS}
R. Anstee, S. Salazar, Forbidden Berge hypergraphs, Electron. J. Combin. 24 (2017), 1.59.

\bibitem{BGKKP}
M. Balko, D. Gerbner, D. Kang, Y. Kim, C. Palmer, Hypergraph based Berge hypergraphs, Graphs Combin. 38 (2022), 11.

\bibitem{B}
C. Berge, Graphs and hypergraphs, North Holland Publishing Company, Amsterdam, 1973.

\bibitem{BGHS}
J. Bermond, A. Germa, M. Heydemann, D. Sotteau, Hypergraphes hamiltoniens, Probl\`emes combinatoires et th\'eorie des graphes, Colloq. Internat. CNRS, Univ. Orsay, Orsay, 1976, 260 (1978), 39--43.

\bibitem{CEP}
D. Clemens, J. Ehrenm\"uller, Y. Person, A Dirac-type theorem for Berge cycles in random hypergraphs, Electron. J. Combin. 27 (2020), 3.39.

\bibitem{CP}
M. Coulson, G. Perarnau, A rainbow Dirac's theorem, SIAM J. Discrete Math. 34 (2020), 1670--1692.

\bibitem{D}
G. Dirac, Some theorems on abstract graphs, Proc. London Math. Soc. (3) 2 (1952), 69--81.

\bibitem{DGS}
P. Dorbec, S. Gravier, G. S\'ark\"ozy, Monochromatic Hamiltonian $t$-tight Berge-cycles in hypergraphs, J. Graph Theory 59 (2008), 34--44.

\bibitem{DF}
A. Dudek, M. Ferrara, Extensions of results on rainbow {H}amilton cycles in uniform hypergraphs, Graphs Combin. 31 (2015), 577--583.

\bibitem{DFR}
A. Dudek, A. Frieze, A. Ruci\'nski, Rainbow Hamilton cycles in uniform hypergraphs, Electron. J. Combin. 19 (2012), 1.46.

\bibitem{FKL}
Z. F\"uredi, A. Kostochka, R. Luo, Berge cycles in non-uniform hypergraphs, Electron. J. Combin. 27 (2020), 3.9.

\bibitem{GP}
D. Gerbner, C. Palmer, Extremal results for Berge-hypergraphs, SIAM J. Discrete Math. 31 (2017), 2314--2327.

\bibitem{GSS1}
A. Gy\'arf\'as, G. S\'ark\"ozy, E. Szemer\'edi, Long monochromatic Berge cycles in colored 4-uniform hypergraphs, Graphs Combin. 26 (2010), 71--76.

\bibitem{GSS2}
A. Gy\'arf\'as, G. S\'ark\"ozy, E. Szemer\'edi, Monochromatic Hamiltonian 3-tight Berge cycles in 2-colored 4-uniform hypergraphs, J. Graph Theory 63 (2010), 288--299.

\bibitem{GLSZ}
E. Gy\H{o}ri, N. Lemons, N. Salia, O. Zamora, The structure of hypergraphs without long Berge cycles, J. Combin. Theory Ser. B 148 (2021), 239--250.

\bibitem{GS}
E. Gy\H{o}ri, N. Salia, Linear three-uniform hypergraphs with no Berge path of given length, J. Combin. Theory Ser. B 171 (2025), 36--48.

\bibitem{GSZ}
E. Gy\H{o}ri, N. Salia, O. Zamora, Connected hypergraphs without long Berge-paths, European J. Combin. 96 (2021), 103353.

\bibitem{HM}
A. Halfpapa, V. Magnan, Positive co-degree thresholds for spanning structures, Electron. J. Combin. 33 (2026), 1.48.

\bibitem{KK}
G. Katona, H. Kierstead, Hamiltonian chains in hypergraphs, J. Graph Theory 30 (1999), 205--212.

\bibitem{KLM}
A. Kostochka, R. Luo, G. McCourt, Dirac-type theorems for long Berge cycles in hypergraphs, J. Combin. Theory Ser. B 168 (2024), 159--191.

\bibitem{LW1}
L. Lu, Z. Wang, On the cover Ramsey number of Berge hypergraphs, Discrete Math. 343 (2020), 111972.

\bibitem{LW}
L. Lu, Z. Wang, On Hamiltonian Berge cycles in $[3]$-uniform hypergraphs, Discrete Math. 344 (2021), 112462.

\bibitem{LW2}
L. Lu, Z. Wang, On the cover Tur\'an number of Berge hypergraphs, European J. Combin. 98 (2021), 103416.

\bibitem{MHG}
Y. Ma, X. Hou, J. Gao, A Dirac-type theorem for uniform hypergraphs, Graphs Combin. 40 (2024), 76.

\bibitem{MO}
L. Maherani, G. Omidi, Monochromatic Hamiltonian Berge-cycles in colored hypergraphs, Discrete Math. 340 (2017), 2043--2052.

\bibitem{RRRSS}
C. Reiher, V. R\"odl, A. Ruci\'nski, M. Schacht, E. Szemer\'edi, Minimum vertex degree condition for tight Hamiltonian cycles in 3-uniform hypergraphs, Proc. Lond. Math. Soc. (3) 119 (2019), 409--439.

\bibitem{RR}
V. R\"odl, A. Ruci\'nski, Dirac-type questions for hypergraphs---a survey (or more problems for Endre to solve), Bolyai Soc. Math. Stud. 21 (2010), 561--590.

\bibitem{Z}
Y. Zhao, Recent advances on Dirac-type problems for hypergraphs, Recent trends in combinatorics 159 (2016), 145--165.
\end{thebibliography}
\end{document}